\documentclass[12pt, reqno]{amsart}
\usepackage{amsmath, amsthm, amscd, amsfonts, amssymb, graphicx, color}
\usepackage[bookmarksnumbered, colorlinks, plainpages]{hyperref}
\hypersetup{colorlinks=true,linkcolor=red, anchorcolor=green, citecolor=cyan, urlcolor=red, filecolor=magenta, pdftoolbar=true}

\newtheorem{theorem}{Theorem}[section]

\newtheorem{corollary}[theorem]{Corollary}
\theoremstyle{definition}
\newtheorem{definition}[theorem]{Definition}
\newtheorem{example}[theorem]{Example}
\newtheorem{exercise}[theorem]{Exercise}

\newtheorem{conjecture}[theorem]{Conjecture}

\theoremstyle{remark}
\newtheorem{remark}[theorem]{Remark}
\numberwithin{equation}{section}

\begin{document}
\setcounter{page}{1}

\title[Langlands program in brief]{Langlands program in brief}

\author[Nikolaev]
{Igor ~V. ~Nikolaev$^1$}

\address{$^{1}$ Department of Mathematics and Computer Science, St.~John's University, 8000 Utopia Parkway,  
New York,  NY 11439, United States.}
\email{\textcolor[rgb]{0.00,0.00,0.84}{igor.v.nikolaev@gmail.com}}





\begin{abstract}
This is  a  credit  mini-course   prepared for a Summer School at the 
University of Sherbrooke.   The course consists of three one-and-half hour lectures
and three credit exercises for a class of advanced graduate students.          

\end{abstract} \maketitle

\tableofcontents

\section{Langlands Program}
\subsection{Fermat's Grand Theorem}
Problem: How to find solutions to the equation
\begin{equation}\label{eq1}
X^n+Y^n=Z^n,
\end{equation}
where $n\in \mathbf{Z}$ and $n\ge 2$?

If  $X, Y, Z\in \mathbf{C}$ 
are complex numbers, it is obvious: 
\begin{equation}
(X,~Y,   ~^n\sqrt{X^n+Y^n}).  
\end{equation}

Suppose we are looking for a solution in the field of rational numbers $\mathbf{Q}$.  So, this is not trivial. 
(The root of an integer is not always an integer!) 
\begin{theorem}  {\bf (Grand Theorem of Pierre de Fermat,    1601 -- 1665)}
If  $n \ge 3$ there are no solutions $(X, Y, Z)$ of  (\ref{eq1}) in integer numbers except  for the trivial one $XYZ=0$.
\end{theorem}
\begin{proof}
(In short:   {\bf  G.~Frey, R.~Taylor, K.~Ribet  and  A.~Wiles}.)
Let us consider the  {\it elliptic curve}  $E$  (will be explained) given by homogeneous equation: 
\begin{equation}
E:   ~Y^2Z = X (X - a^pZ) (X +b^pZ) 
\end{equation}
such that:
\begin{equation}
a^p - b^p = c^p,
\end{equation}
where $a, b, c \in \mathbf{Z}$  and $p$ is a prime number.  
Then ({\bf G. Frey}) $E$ is not a {\it modular curve} (will be explained).
So, 
\begin{equation}
a^p - b^p \neq c^p
\end{equation}
and Fermat's Grand Theorem is proven! 
\end{proof}

\subsection{Analysis: Modular group, modular forms, etc.}
Modular forms are a generalization of periodic functions. 
\begin{example}
Periodic function: 

\medskip
(i)  $\sin (x + 2\pi) = \sin ~x, ~\forall x \in \mathbf{C}$

\smallskip
(ii)  $e^{(2\pi + x)i} = e^{ix}, ~\forall x \in \mathbf{R}$.
\end{example}
\begin{definition}
Modular group:
\begin{equation}
SL_2(\mathbf{Z})=\left\{\left(
\begin{matrix}
a &b\cr c & d
\end{matrix}
\right) \quad | \quad a,b,c,d\in \mathbf{Z},  \quad ad-bc=1\right\}.
\end{equation}
\end{definition}
\begin{exercise} 
Check that  $SL_2(\mathbf{Z})$ is a multiplicative group;   find the unit and inverse elements!
\end{exercise}
\begin{definition}
Congruence group: 
\begin{equation}
\Gamma_0(N)=\left\{\left(
\begin{matrix}
a &b\cr c & d
\end{matrix}
\right)\in  SL_2(\mathbf{Z})   \quad | \quad c\equiv 0 ~mod~N\right\},
\end{equation}
where $N \ge 1$ is an integer. 
\end{definition}
\begin{exercise} 
Prove that $\Gamma_0(N)$ is a multiplicative group! 
\end{exercise}

\bigskip\noindent
If  ${\Bbb H}:=\{z\in\mathbf{C}~|~\Im ~(z)>0\}$ 
is a hyperbolic half-plane, then $SL_2(\mathbf{Z})$ acts on  ${\Bbb H}$ by the formula:
\begin{equation}
z\mapsto {az+b\over cz+d}, \quad  \left(
\begin{matrix}
a &b\cr c & d
\end{matrix}
\right)\in  SL_2(\mathbf{Z}).   
\end{equation}
The subgroup  $\Gamma_0(N)\subset SL_2(\mathbf{Z})$ acts on  ${\Bbb H}$ and the space
$X_0(N):= {\Bbb H}/\Gamma_0(N)$  is a  {\it Riemann surface } of genus  $g \ge 0$.  
\begin{example}
Each  $N$ defines the genus $g$ of $X_0(N)$;  for instance:

\medskip
(i) if $N = 2$,  then  $X_0(2)$ is a sphere, i.e. $g = 0$;

\smallskip
(ii) if  $N = 11$,  then  $X_0(11)$  is a torus, i.e. $g = 1$.  
\end{example}

\begin{definition}
An automorphic function on ${\Bbb H}$ with respect to the group $\Gamma_0(N)$ 
is an analytical function $f: {\Bbb H}\to \mathbf{C}$,  
such that: 
\begin{equation}
f\left({az+b\over cz+d}\right)=f(z), \quad \forall z\in {\Bbb H}, \quad \forall  \left(
\begin{matrix}
a &b\cr c & d
\end{matrix}
\right)\in \Gamma_0(N). 
\end{equation}
A {\it modular form} of weight equal to $2$  is a function  $f: {\Bbb H}\to \mathbf{C}$,
 such that: 
\begin{equation}
f\left({az+b\over cz+d}\right)=(cz+d)^2 f(z), \quad \forall z\in {\Bbb H},  \quad \forall  \left(
\begin{matrix}
a &b\cr c & d
\end{matrix}
\right)\in \Gamma_0(N). 
\end{equation}
\end{definition}

\subsection{Arithmetic: Elliptic curves, rational points, etc.}
\begin{definition}
An elliptic curve is a cubic of the form: 
\begin{equation}\label{eq1.11}
y^2 = x (x - 1) (x -\lambda), \quad \lambda\in \mathbf{C}.
\end{equation}
\end{definition}
\begin{example}
If $\lambda \in \mathbf{R}$, then we can draw a graph of the elliptic curve, as shown in Figure 1. 
 \end{example}
\begin{figure}
\begin{picture}(300,100)(0,0)

\put(100,50){\line(1,0){60}}
\put(130,20){\line(0,1){60}}

\thicklines
\qbezier(165,25)(115,50)(165,75)
\qbezier(110,50)(117,15)(130,50)
\qbezier(110,50)(117,85)(130,50)

\end{picture}
\caption{Affine cubic $y^2=x(x-1)(x+1)$.}
\end{figure}

\begin{example} {\bf (Weierstrass)}
If  $\lambda \in \mathbf{C}$, then the elliptic curve is a Riemann surface of the genus $g = 1$, i.e. 
a compex torus. 
\end{example}
\begin{definition}
If  $\lambda\in\mathbf{Q}$ then the elliptic curve  (\ref{eq1.11}) is called rational. 
Notation:  $E(\mathbf{Q})$. 
\end{definition}

\subsection{Conjecture of Shimura-Taniyama}
\begin{theorem}  {\bf (Wiles, conjectured by Shimura-Taniyama)}
For each rational elliptic curve $E(\mathbf{Q})$, there exists an integer $N> 1$,
 such that there is a holomorphic map between two Riemann surfaces: 
\begin{equation}\label{eq1.12} 
X_0(N)\to E(\mathbf{Q}).
\end{equation}
\end{theorem}
\begin{definition}
If $E(\mathbf{Q})$ satisfies condition  (\ref{eq1.12}) then $E(\mathbf{Q})$ is called modular. 
\end{definition}
\begin{corollary}\label{cr1}
{\bf (A.~Wiles)}
If an elliptic curve is rational, then such a curve is modular. 
 \end{corollary}
How can we prove Fermat's Grand Theorem using Corollary  \ref{cr1}?
\begin{theorem}  {\bf (Fermat's Grand Theorem)}
If  $n \ge 3$ there are no solutions $(X, Y, Z)$ of  (\ref{eq1}) in integer numbers except  for the trivial one $XYZ=0$.
\end{theorem}
\begin{proof}
Suppose $a^p - b^p = c^p$ is a non-trivial solution for a prime number $p> 2$. 
Let us consider the rational elliptic curve $E(\mathbf{Q})$ of form: 
\begin{equation}
Y^2Z = X (X - a^pZ) (X + b^pZ). 
\end{equation}
A reduction modulo $p$ argument leads to the conclusion that $E(\mathbf{Q})$ is never modular. 
This is a contradiction with Corollary \ref{cr1} ! 
Therefore $a^p - b^p \ne c^p$ 
\end{proof}

\subsection{What is the Langlands Program?} 
Robert Langlands, born in British Columbia, professor at Princeton University (Institute of Advanced Studies). 
\begin{remark}
Hyperbolic half-plane  ${\Bbb H}$ is a homogeneous space of the Lie group  $SL_2(\mathbf{R})$.
\end{remark}
\begin{exercise} 
Prove that 
\begin{equation}
{\Bbb H} \cong SL_2(\mathbf{R}) /  SO_2(\mathbf{R}),
\end{equation}
i.e. ${\Bbb H}$ is a homogeneous space. 
(Hint: Check that each  $z \in {\Bbb H}$ is a fixed point of  $SO_2(\mathbf{R})$ 
and choose $z = i$.)
\end{exercise}

So, we have a holomorphic map: 
\begin{equation}\label{eq1.15}
\underbrace{\Gamma_0(N)\backslash SL_2(\mathbf{R}) /  SO_2(\mathbf{R})}_{Lie  ~groups}
\longrightarrow 
\underbrace{E(\mathbf{Q})}_{arithmetic}.  
\end{equation}

\bigskip
\centerline{\underline{Langlands program in brief}.}

\medskip
Motivated by  (\ref{eq1.15}),  find a relation  between:

\bigskip   
\displaymath
\fbox{\begin{minipage}{8em}
Representation Theory of the Lie groups
(like $SL_2(\mathbf{R})$)
\end{minipage}}
\qquad
\leftrightarrow
\qquad
\fbox{\begin{minipage}{10em}
Number Theory  and  Arithmetic
(like $E(\mathbf{Q})$ et Fermat's Grand Theorem)
\end{minipage}}
\enddisplaymath

\bigskip
\centerline{\underline{Why}?}

\bigskip
For example, to prove the Grand Theorem of Fermat! (among other things).

\bigskip 
In lectures 2 and 3 we will detail the Langlands Program.


\section{Analysis of Modular Forms}
\subsection{Automorphic functions}
Automorphic functions are a generalization of periodic functions. 
\begin{example}
Periodic function: 

\medskip
(i)  $\sin (x + 2\pi) = \sin ~x, ~\forall x \in \mathbf{C}$

\smallskip
(ii)  $e^{(2\pi + x)i} = e^{ix}, ~\forall x \in \mathbf{R}$.
\end{example}
\begin{definition}
 An automorphic function on  ${\Bbb H}$ with respect to the group $SL_2(\mathbf{Z})$ is an analytical
  function  $f: {\Bbb H}\to \mathbf{C}$, such that: 
\begin{equation}
f\left({az+b\over cz+d}\right)=f(z), ~\forall z\in {\Bbb H}, ~\forall  \left(
\begin{matrix}
a &b\cr c & d
\end{matrix}
\right)\in SL_2(\mathbf{Z}). 
\end{equation}
\end{definition}

\subsection{Modular forms of weight  $2k$}
\begin{definition}
 A modular form of weight $2k$ with respect to the group  $\Gamma_0(N)$ is a holomorphic function 
 $f: {\Bbb H}\to \mathbf{C}$
such that: 
\begin{equation}
f\left({az+b\over cz+d}\right)=(cz+d)^{2k} f(z), ~\forall z\in {\Bbb H}, ~\forall  \left(
\begin{matrix}
a &b\cr c & d
\end{matrix}
\right)\in \Gamma_0(N). 
\end{equation}
\end{definition}
\begin{remark}\label{rmk2.4}
Each modular form of weight $2k$  corresponds to a holomorphic $k$-form on the Riemann 
surface $X_0(N)$. In particular, if $k = 1$ we have a bijection between the modular forms of weight $2$ 
and holomorphic differentials on $X_0(N)$. 
\end{remark}

\subsection{Cusp points  of  the group $\Gamma_0(N)$}
\begin{definition}
Point $x\in\partial {\Bbb H}$ (an ``absolute''  of the half-plane ${\Bbb H}$) 
is called a cusp point  of the group  $\Gamma_0(N)$ if there exists  $\alpha\in\Gamma_0(N)$
such that $\alpha(x)=x$.  
\end{definition}
\begin{example} 
$${ax+b\over cx+d}=x\quad \Longleftrightarrow\quad cx^2+(d-a)x-b=0,  $$
$$x_{1,2}={a-d\pm \sqrt{(d-a)^2+4bc}\over 2c}= {a-d\pm \sqrt{(a+d)^2-4(ad-bc)}\over 2c}=$$
$$=  {a-d\pm \sqrt{(a+d)^2-4}\over 2c}.$$
But  $x_1=x_2=x$,  therefore $|a+d|=2$ and $x={a-d\over 2c}\in\partial {\Bbb H}$ 
is a unique fixed point of the transformation:
$$
 \alpha=\left(
 \begin{matrix}
 a & b \cr
 c & -a\pm 2
 \end{matrix}
 \right)
 \in \Gamma_0(N). 
 $$  
 In particular,   cusp points of the group  $\Gamma_0(N)$ are 
 the  rational points  at the absolute $\partial {\Bbb H}$. 
 \end{example}

\bigskip
\centerline{\underline{Why cusp points are important}?}

\bigskip\noindent
Consider the Riemann surface:
\begin{equation}
X_0(N)={\Bbb H}/\Gamma_0(N).
\end{equation}
The orbits of  group $\Gamma_0(N)$ have a “deficiency” in the cusp point $x$, so the Riemann surface $X_0(N)$   
has a “cusp” in $x$!

\subsection{Cusp modular forms}
Suppose that $f(z)$ is a modular form of weight $2$ for the group
$\Gamma_0(N)$, i.e.
\begin{equation}
f\left({az+b\over cz+d}\right)=(cz+d)^{2} f(z), ~\forall z\in {\Bbb H}, ~\forall  \left(
\begin{matrix}
a &b\cr c & d
\end{matrix}
\right)\in \Gamma_0(N). 
\end{equation}
\begin{definition}
If $f(x)=0$ at each parabolic point $x\in \partial {\Bbb H}$ of group $\Gamma_0(N)$, then $f(z)$ is called a 
cusp modular form. The space of all parabolic forms is denoted by $S_2(\Gamma_0(N))$. 
(In German “Spitzform” means cusp form). 
\end{definition}

\bigskip
\centerline{\underline{Why are cusp modular forms important}?}

\bigskip\noindent
  According to Remark \ref{rmk2.4}, each  $f(z)\in S_2(\Gamma_0(N))$ corresponds to 
  the holomorphic differential  $\omega=f(z)dz$ of the Riemann surface $X_0(N)$. 
  So the zeros of $\omega$ coincide with the “cusps” of  $X_0(N)$. 
  (There is always a finite number of such “cusps”.) 
  Also some forms  $f(z)\in S_2(\Gamma_0(N))$ have a very interesting Fourier series,
   see the next paragraph!

\subsection{Fourier series of cusp modular forms}
If  $f(z)\in S_2(\Gamma_0(N))$, then $f(x) = 0$  if and only if $x\in\partial {\Bbb H}$ is a cusp point of $\Gamma_0(N)$. 
We will introduce a variable
\begin{equation}
q=e^{2\pi i z}. 
\end{equation}
\begin{definition}
 Fourier series of a cusp modular  form  $f(z)\in S_2(\Gamma_0(N))$ is a series
\begin{equation}
f(z)=\sum_{n=-\infty}^{n=\infty} c_n q^n.
\end{equation}
\end{definition}
\begin{remark}
The coefficients  $c_n$ of 
certain cusp modular forms $f(z)\in S_2(\Gamma_0(N))$.
encode important arithmetic information, see below.
\end{remark}
\begin{example} {\bf ($j$-invariant of  F.~Klein)}
Each elliptic curve has a $j$-invariant that is constant on the isomorphism class of an elliptic curve. 
The $j$-invariant is a cusp modular  form having  the  Fourier series: 
\begin{equation}
j(z)={1\over q}+744+196884 q + 21493760 q^2+\dots
\end{equation}
 Small miracle: the coefficients $c_n$ of  $j(z )$  are connected to the order 
 of finite simple groups (called “sporadic”)! 
 \end{example}

\subsection{$L$-series of cusp modular forms}
\begin{definition}
If $f(z)\in S_2(\Gamma_0(N))$ et  $f(z)=\sum_{n=-\infty}^{n=\infty} c_n q^n$, then
 the convergent series
\begin{equation}
L(s,f):=\sum_{n=1}^{\infty} {c_n\over n^s}
\end{equation}
is called an $L$-series of the cusp modular form. 
\end{definition}
\begin{remark}
The series  $L(s,f)$ is a generalization of the Riemann zeta function: 
\begin{equation}
\zeta(s)=\sum_{n=1}^{\infty} {1\over n^s}. 
\end{equation}
\end{remark}

\section{Arithmetic of rational elliptic curves }
\subsection{Rational elliptic curves}
\begin{definition}
A rational elliptic curve is a cubic of  the form:
\begin{equation}
E(\mathbf{Q}):=\{X,Y,Z\in \mathbf{C}P^2 ~|~Y^2Z=X(X-Z)(X-\lambda Z),\quad \lambda\in\mathbf{Q}\}.
\end{equation}
Rational points of $E(\mathbf{Q})$ correspond to the integer  triples $(X, Y, Z)$.
\end{definition}

\subsection{Finite fields}
\begin{exercise} 
Give at least 4 different examples of fields of characteristic  zero. Reminder: 
The characteristic of a field $F$ is the minimal number $n$ such that: 
\begin{equation}
nx:=\underbrace{x+x+\dots+x}_{n ~ times}=0
\end{equation}
for every $x\in F$.  If $n$ does not exist, then $char~(F):=0$. 
\end{exercise}
\begin{example} 
If  $p\ge 2$ is a prime number, then the  field  ${\Bbb F}_p$  with a finite number elments
has the characteristic  $char~({\Bbb F}_p)=p$. Proof:
\begin{equation}
px:=\underbrace{x+x+\dots+x}_{p ~ times}=0 ~mod~ p, \qquad\forall x\in {\Bbb F}_p. 
\end{equation}
 \end{example}

\subsection{Reduction  of  $E(\mathbf{Q})$ modulo $p$}
\begin{remark}
If $(X, Y, Z)$ is a rational point of the elliptic curve $E(\mathbf{Q})$, then $(X, Y, Z)$ are also a point of the 
reduced elliptic curve $E(\mathbf{Q})~mod~p$, i.e.  solution of the equation: 
\begin{equation}\label{eq3.4}
Y^2Z=X(X-Z)(X-\lambda Z) ~mod~p. 
\end{equation}
\end{remark}
\begin{definition}
The elliptic curve $E({\Bbb F}_p)$ defined by equation (\ref{eq3.4}) 
 is called reduction of $E(\mathbf{Q})$ modulo $p$.
\end{definition}
\begin{remark}
The number of solutions of (\ref{eq3.4}) is always finite, i.e.  
\begin{equation}
|E({\Bbb F}_p)|<\infty.
\end{equation}
\end{remark}

\subsection{Zeta function of $E({\Bbb F}_p)$}
\begin{definition}
The zeta function of the elliptic curve  $E({\Bbb F}_p)$ is 
\begin{equation}
Z(u, E({\Bbb F}_p))=\exp \left(\sum_{n=1}^{\infty} {|E({\Bbb F}_{p^n})|\over n} ~u^n\right),
\end{equation}
where  ${\Bbb F}_{p^n}$ is  an extension of degree $n$ of the field  ${\Bbb F}_p$. 
\end{definition}
\begin{remark}
It is not difficult to prove (using the Lefschetz trace formula) that 
$Z(u, E({\Bbb F}_p))$ is always convergent and:
\begin{equation}
Z(u, E({\Bbb F}_p))={1\over 1-a_pu+pu^2},
\end{equation}
where  $a_p=p+1-|E({\Bbb F}_p)|$. 
\end{remark}

\subsection{$L$-function of  $E(\mathbf{Q})$}
\begin{definition}
 By an $L$-function of $E(\mathbf{Q})$ one understands the infinite product: 
\begin{equation}
L(s, E(\mathbf{Q}))=\prod_p  Z(p^{-s}, E({\Bbb F}_p))=
\prod_p {1\over 1-a_p p^{-s} +p^{1-2s}}.
\end{equation}
\end{definition}
\begin{remark}{\bf (Birch and  Swinnerton-Dyer)}
The $L$-function of $E(\mathbf{Q})$  encodes  “all” arithmetic information 
about the rational elliptic curve $E(\mathbf{Q})$. For example, the order of zero at point $s = 1$ is equal to the rank of $E(\mathbf{Q})$ 
(Birch and Swinnerton-Dyer conjecture). 
\end{remark}

\subsection{Eichler-Shimura Theorem}
We know that there is a holomorphic map: 
\begin{equation}\label{eq3.9}
X_0(N)\longrightarrow E(\mathbf{Q}).
\end{equation}
On the other hand, there is an $L$-function $L(f, s)$ attached to the cusp  modular form $f(z)\in S_2(\Gamma_0(N))$ on 
the Riemann surface $X_0(N)$..  According to  (\ref{eq3.9}), we can ask  how $L (f, s)$  is related to the $L$-function  $L(E(\mathbf{Q}), s)$ 
attached to the rational elliptic curve $E(\mathbf{Q})$? 
\begin{theorem}  {\bf (Eichler and Shimura)}
For $N> 1$ there exists a cusp modular form $f(z)\in S_2(\Gamma_0(N))$ and a rational elliptic curve  $E(\mathbf{Q})$, such that
\begin{equation}\label{eq3.10}
L(f,s)\equiv L(E(\mathbf{Q}), s).
\end{equation}
\end{theorem}
\begin{remark}
The cusp modular form $f$ in equation (\ref{eq3.10}) is called a {\it Hecke eigenform}; 
such a form is unique and invariant with respect to an algebra of Hecke operators. 
\end{remark}

\subsection{The automorphic and motivic $L$-functions}
\begin{definition}
The $L$-function $L (f, s)$ is called {\it automorphic}, because it comes from the cusp modular forms attached to the 
homogeneous space of the Lie group  $SL_2(\mathbf{R})$. 
\end{definition}
\begin{definition}
The $L$-function $L(E(\mathbf{Q}),s)$ is called {\it motivic}, because it comes from the 
Lefschets trace formula attached to a motivic cohomology of  the variety  $E(\mathbf{Q})$; see {\bf Alexandre Grothendieck}. 
\end{definition}

\subsection{Langlands Conjectures}
The formula (\ref{eq3.10})  can be generalized to the arbitrary arithmetic varieties. 
One of Robert ~P. ~Langlands Conjectures is a deep and spectacular generalization in this direction. 
 \begin{conjecture}\label{cnj3.15}
 {\bf (R.~P.~Langlands)}
 Each motivic $L$-function is equal to a product of the automorphic $L$-functions for certain Lie groups 
 (called reductive algebraic groups). 
 \end{conjecture}
\begin{exercise} 
Prove  \ref{cnj3.15} !!!
\end{exercise}

\bigskip\noindent
{\sf Acknowlegment.} 
I thank Prof.  Ibrahim Assem for an opportunity to participate  in  the Summer School 
at the University of  Sherbrooke.

\bibliographystyle{amsplain}

\end{document}